\documentclass[graybox]{svmult}

\usepackage{type1cm}
\usepackage{makeidx} 
\usepackage{multicol}
\usepackage[bottom]{footmisc}
\usepackage{newtxtext}  
\usepackage{newtxmath}  

\usepackage{mathrsfs}

\makeindex           

\begin{document}

\title*{Harmonic and Trace Inequalities\\ in Lipschitz Domains\thanks{This 
is a preprint of a paper accepted 15-Nov-2018 and
whose final and definite form is a book chapter at Springer New York, 
on the topic of ``Frontiers in Functional Equations and Analytic Inequalities'',
Edited by G. Anastassiou and J. Rassias.}}
\titlerunning{Harmonic and Trace Inequalities in Lipschitz Domains} 

\author{Soumia Touhami, Abdellatif Chaira and Delfim F. M. Torres}
\authorrunning{S.~Touhami, A.~Chaira and D.~F.~M.~Torres} 

\institute{Soumia Touhami \at 
Universit\'{e} Moulay Ismail, Facult\'{e} des Sciences,\\ 
Laboratoire de Math\'{e}matiques et leures Applications,\\ 
Equipe EDP et Calcul Scientifique, B.P. 11201 Zitoune, 
50070 Mekn\`{e}s, Morocco\\ 
\email{touhami16soumia@gmail.com}
\and 
Abdellatif Chaira \at 
Universit\'{e} Moulay Ismail, Facult\'{e} des Sciences,\\ 
Laboratoire de Math\'{e}matiques et leures Applications,\\ 
Equipe EDP et Calcul Scientifique, B.P. 11201 Zitoune, 
50070 Mekn\`{e}s, Morocco\\ 
\email{a.chaira@fs.umi.ac.ma}
\and 
Delfim F. M. Torres (corresponding author)
\at Center for Research and Development in Mathematics and Applications (CIDMA),\\ 
Department of Mathematics, University of Aveiro, 3810-193 Aveiro, Portugal\\ 
\email{delfim@ua.pt}}

\maketitle


\abstract{We prove boundary inequalities 
in arbitrary bounded Lipschitz domains
on the trace space of Sobolev spaces. 
For that, we make use of the trace operator, its Moore--Penrose inverse, 
and of a special inner product. We show that our trace inequalities are 
particularly useful to prove harmonic inequalities, which serve
as powerful tools to characterize the harmonic functions  
on Sobolev spaces of non-integer order.}

\bigskip


\medskip



\section{Introduction}

In this article we establish some new and important operator 
inequalities connected with traces on Hilbert spaces.
Trace inequalities find several interesting applications,
e.g., to problems from quantum statistical mechanics 
and information theory \cite{MR2681769,MR3592581,MR3868111}.
Here we establish new trace inequalities in Lipschitz domains,
that is, in a domain of the Euclidean space whose boundary is 
``sufficiently regular'', in the sense that it can be thought of as, 
locally, being the graph of a Lipschitz continuous function \cite{MR3766365}. 
The study of Lipschitz domains is an important research area \emph{per se}, 
since many of the Sobolev embedding theorems require them as the natural domain 
of study \cite{MR3677864}. Consequently, many partial differential equations 
found in applications and variational problems are defined on Lipschitz domains 
\cite{MR3865109,MR3288348,MR3861903}. In our case, we investigate
the application of the obtained trace inequalities in Lipschitz domains
to harmonic functions \cite{MR1625845}, which is a subject of strong current research
\cite{MR3813463,MR2919427,MR0856473,MR3455304}.

The paper is organized as follows.
In Section~\ref{sec:02}, we fix notations
and recall necessary definitions and results,
needed in the sequel. Our contribution is then 
given in Section~\ref{sec:3}: we prove a
Moore--Penrose inverse equality (Theorem~\ref{prop:perm}),
trace inequalities (Theorem~\ref{thm:1}), and harmonic inequalities 
(Theorems~\ref{prop:1} and \ref{prop:3}). As
an application of Theorem~\ref{prop:3}, we 
obtain a functional characterization 
of the harmonic Hilbert spaces for the range of values 
$0\leq s\leq 1$ (Corollary~\ref{cor:spect:caract}).


\section{Preliminaries}
\label{sec:02}

Let $\mathcal H_1$ and $\mathcal H_2$ be two Hilbert spaces with inner products 
$(\cdot,\cdot)_{\mathcal H_1}$ and $(\cdot,\cdot)_{\mathcal H_2}$ and associated 
norms $\|\cdot\| _{\mathcal H_1}$ and $\|\cdot\| _{\mathcal H_2}$, respectively.  
We denote by $\mathcal L(\mathcal H_1, \mathcal H_2)$
the space of all linear operators from $\mathcal H_1$ 
into $\mathcal H_2$ and $\mathcal L(\mathcal H_1, \mathcal H_1)$ 
is briefly denoted by $\mathcal L(\mathcal H_1)$. For an operator 
$A\in \mathcal L(\mathcal H_1,\mathcal H_2)$, 
$\mathcal{D}(A)$, $\mathcal R(A)$ and $\mathcal N(A)$ 
denote its domain, its range, and its null space, respectively. 
The set of all bounded operators from $\mathcal H_1$ 
into $\mathcal H_2$ is denoted by $\mathcal B(\mathcal H_1, \mathcal H_2)$, 
while $\mathcal B(\mathcal H_1, \mathcal H_1)$ is briefly denoted 
by $\mathcal B(\mathcal H_1)$. The set of all closed densely defined 
operators from $\mathcal H_1$ into $\mathcal H_2$ is denoted 
by $\mathcal C(\mathcal H_1, \mathcal H_2)$ and, analogously as before,
$\mathcal C(\mathcal H_1, \mathcal H_1)$ is denoted by
$\mathcal C(\mathcal H_1)$. For $A \in \mathcal C(\mathcal H_1,\mathcal H_2)$, 
its adjoint operator is denoted by $A^*\in \mathcal C(\mathcal H_2,\mathcal H_1)$. 

The Moore--Penrose inverse of a closed densely defined operator 
$A \in {\mathcal C}({\mathcal H}_1, {\mathcal H}_2)$, 
denoted by $A^\dagger$, is defined as the unique linear operator 
in ${\mathcal C }({\mathcal H}_2, {\mathcal H}_1)$ such that
$$
{\mathcal D}(A^\dagger)={\mathcal R}(A)
\oplus {\mathcal N}(A^*), 
\quad  \mathcal N(A^{\dagger })= \mathcal N(A^*), 
$$ 
and 
\begin{equation*}
\begin{cases}
AA^\dagger A=A, &\text{  }  \\
A^\dagger AA^\dagger=A^\dagger, 
&\text{}
\end{cases}
\qquad
\begin{cases}
AA^\dagger \subset P_{\overline{{\mathcal R}(A)}}, 
& \text{  }  \\
A^\dagger A \subset P_{\overline{{\mathcal R}( A^\dagger)}}, 
&\text{}
\end{cases}
\end{equation*}
where $\overline{\mathcal{E}}$ denotes the closure of $\mathcal{E}$,
$\mathcal{E} \in \left\{{\mathcal R}(A), {\mathcal R}( A^\dagger) \right\}$,
and $P_{\overline{\mathcal{E}}}$ the orthogonal projection 
on the closed subspace $\overline{\mathcal{E}}$. 
The following lemma is used in the proof of our 
Moore--Penrose inverse equality (Theorem~\ref{prop:perm}).

\begin{lemma}[See Lemma~2.5 and Corollary~2.6 of \cite{L}]  
\label{eq:lemma2.5Cor2.6}
Let $A \in \mathcal C(\mathcal H_1, \mathcal H_2)$ 
and $B \in \mathcal C(\mathcal H_2,\mathcal H_1)$ be
such that $ B=A^{\dagger}$. Then,
\begin{enumerate}
\item $A(I+A^*A)^{-1} = B^*(I+BB^*)^{-1}$;
\item $(I+A^*A)^{-1}+(I+BB^*)^{-1}= I+P_{\mathcal N(B^*)}$;     
\item $A^*(I+AA^*)^{-1} = B(I+B^*B)^{-1}$;
\item $(I+AA^*)^{-1}+(I+B^*B)^{-1}= I+P_{\mathcal N(A^*)}$;
\item $(I+AA^*)^{-1}+(I+B^*B)^{-1}= I$  (if $A^*$ is injective);
\item $\mathcal N(A^*(I+AA^*)^{-1/2}) =\mathcal N(A^*) = \mathcal N(B)$.
\end{enumerate}
\end{lemma}

\begin{lemma}[See Theorem~3.5 of \cite{TCT}] 
\label{thm:3.5}
Let ${\mathcal H}_1$ and ${\mathcal H}_2$ be two Hilbert spaces, 
$A \in {\mathcal B}({\mathcal H}_1, {\mathcal H}_2)$, and 
$B$ be its Moore--Penrose inverse. Then, the operator $B^*(I+BB^*)^{-1/2}$ 
is bounded with closed range and has a bounded Moore--Penrose inverse given by 
\begin{equation*}
T_B=B(I+B^*B)^{-1/2}+A^*(I+B^*B)^{-1/2}.
\end{equation*}
Moreover, the adjoint operator of $T_B$ is $T_{B^*}$, where 
\begin{equation*}
T_{B^*}=B^*(I+BB^*)^{-1/2}+A(I+BB^*)^{-1/2}.
\end{equation*}
\end{lemma}

\begin{lemma}[See Theorem~3.8 of \cite{TCT}] 
\label{lemme:1}
Let $\mathcal H_1$ and $\mathcal H_2$ be two Hilbert spaces, 
$A\in \mathcal B(\mathcal H_1,\mathcal H_2)$, and $B$ be 
its Moore--Penrose inverse. Then, the decomposition
$$
A= (I+B^*B)^{-1/2} T_{B^*}
$$
holds, where $T_{B^*}=B^*(I+BB^*)^{-1/2}+A(I+BB^*)^{-1/2}$.
\end{lemma}

\begin{lemma} [See Corollary~3.7 of \cite{TCT}] 
\label{lem:2}
Let $A\in\mathcal B(\mathcal H_1,\mathcal H_2)$ and $B$ be 
its Moore--Penrose inverse. Then, $T_B$ is an isomorphism from 
$\mathcal{R}(B^*)$ to $\mathcal{N}(B^*)^{\perp}$, where
$\mathcal{N}(B^*)^{\perp}$ denotes the orthogonal 
complement of $\mathcal{N}(B^*)$.
\end{lemma}
 
Let $\Omega$ be an open subset of $\mathbb R^d$ with boundary 
$\partial \Omega$ and closure $\overline{\Omega}$. 
We say that $\partial \Omega$ is Lipschitz continuous 
if for every $x\in \partial \Omega$ there exists a coordinate system 
$(\widehat{y}, y_d)\in \mathbb{R}^{d-1}\times\mathbb{R}$, 
a neighborhood $Q_{\delta,\delta'}(x)$ of $x$, and a Lipschitz function 
$\gamma_x:\widehat{Q}_{\delta} \rightarrow \mathbb{R}$, with 
the following properties:
\begin{enumerate}
\item $\Omega \cap Q_{\delta,\delta'}(x) 
= \left\{(\widehat{y},y_d) \in Q_{\delta,\delta'}(x) 
\ / \ \gamma_x(\widehat{x}) < y_{d} \right\}$;
  
\item $\partial \Omega \cap  Q_{\delta,\delta'}(x) 
=\left\{  (\widehat{y},y_d) \in Q_{\delta,\delta'}(x) 
\ / \ \gamma_x(\widehat{x}) = y_{d} \right\}$;
\end{enumerate}
where 
$Q_{\delta,\delta'}(x) 
= \left\{  (\widehat{y},y_d) \in \mathbb R^d \ / \  
\|\widehat{y}-\widehat{x}\| _{\mathbb R^{d-1}} < \delta  
\ \ \text{and} \ \ |y_d - x_d | < \delta' \right\}$ 
and 
$$ 
\widehat{Q}_{\delta}(x) = \left\{ \widehat{y} \in \mathbb R^{d-1} 
\ / \  \|\widehat{y}-\widehat{x}\| _{\mathbb R^{d-1}} < \delta  \right\} 
$$
for  $\delta, \delta' > 0$. 
An open connected subset $\Omega \subset \mathbb{R}^d$, 
whose boundary is Lipschitz continuous, 
is called a Lipschitz domain.
In the rest of this paper, $\Omega$ denotes a bounded Lipschitz 
domain in $\mathbb{R}^d$, $d\geq 2$. We denote by 
$\mathcal C^k(\Omega)$, $k\in \mathbb{N}$ 
or $k= \infty$, the space of real $k$ times continuously 
differentiable functions on $\Omega$. The space $\mathcal{C}^{\infty}$
of all real functions on $\Omega$ with a compact support 
in $\Omega$ is denoted by $\mathcal C_c^{\infty} (\Omega)$.
We say that a sequence $(\varphi_n)_{n\geq 1}\in \mathcal C_c^{\infty}(\Omega)$  
converges to $\varphi \in \mathcal C_c^{\infty}(\Omega)$, 
if there exists a compact $Q\subset \Omega$ such that  
$\mathrm{supp}(\varphi_n) \subset Q$ for all $n\geq 1$ 
and, for all multi-index $\alpha \in \mathbb{N}^d$, 
the sequence $(\partial^{\alpha} \varphi_n)_{n\geq 1}$ 
converges uniformly to $\partial ^{\alpha} \varphi$, 
where $\partial^{\alpha}$ denotes the partial derivative of order $\alpha$. 
The space $\mathcal C_c^\infty(\Omega)$, induced by this convergence, 
is denoted by $\mathscr{D}(\Omega)$,
while $\mathscr{D}^{\prime}(\Omega)$ is
the space of distributions on $\Omega$.
For $k\in \mathbb N$, $H^k(\Omega)$ is the space 
of all distributions $u$ defined on $\Omega$ such that 
all partial derivatives of order at most $k$ lie in $L^2(\Omega)$, i.e., 
$\partial ^{\alpha} u \in L^2(\Omega) \ \ 
\forall  \ |\alpha | \leq k$.
This is a Hilbert space with the scalar product
\begin{equation*}
(u,v)_{k,\Omega} = \sum_{|\alpha | \leq k} 
\int_{\Omega} \partial ^{\alpha}u \ \partial^{\alpha}v \ dx, 
\end{equation*}
where $dx$ is the Lebesgue measure and $u,v \in H^k(\Omega)$.
The corresponding norm, denoted by $\|\cdot\|_{k,\Omega}$, is given by
\begin{equation*}
\|u\|_{k,\Omega} = \left(\sum_{|\alpha| \leq k} 
\int_{\Omega} |\partial ^{\alpha}u |^2 \ dx \ \right)^{1/2}.
\end{equation*}
Sobolev spaces $H^s(\Omega)$, for non-integers $s$, 
are defined by the real interpolation method \cite{Ad,Mc,T}. 
The trace spaces $H^s(\partial \Omega)$ can be defined by using 
charts on $\partial \Omega$ and partitions 
of unity subordinated to the covering of $\partial \Omega$. 
If $\Omega$ is a Lipschitz hypograph, then 
there exists a Lipschitz function  
$\gamma: \mathbb{R}^{d-1} \rightarrow \mathbb{R}$ such that 
$\Omega = \left\{ x \in  \mathbb R^{d-1} 
\ / \  x_{d}   <  \gamma (\widehat{x})
\text{ for all } \widehat{x} \in \mathbb R^{d-1}   \right\}$.
This allows to construct Sobolev spaces on the boundary $\partial \Omega$, 
in terms of Sobolev spaces on $\mathbb{R}^{d-1}$ \cite{Mc}.
This is done as follows. For $g\in L^2(\partial \Omega)$, we define
$g_{\gamma}(\widehat{x}) = g(\widehat{x}, \gamma(\widehat{x})) 
\ \mbox{for} \  \widehat{x} \in \mathbb R^{d-1}$, we let
$$
H^s(\partial \Omega) = \left\{ \ g\in L^2(\partial \Omega) \ | 
\ g_{\gamma} \in H^s(\mathbb R^{d-1})  \ \mbox{for} \ 0\leq s\leq 1 \right\}, 
$$
and equip this space with the inner product
$(g,y)_{H^s(\partial \Omega)} 
= (g_{\gamma}, y_{\gamma}) _{s,\mathbb R^{d-1}}$,
where
$$
(u,v)_{s,\mathbb R^{d-1}} = \int_{\mathbb R^{d-1}} (1+|\xi| ^2)^s 
\widehat{u}(\xi) \widehat{v}(\xi) \ d\xi 
$$ 
and $\widehat{u}$ denotes the Fourier transform of $u$.
Recalling that any Lipschitz function is almost everywhere differentiable, 
we know that any Lipschitz hypograph $\Omega$ has a surface measure $\sigma$ 
and an outward unit normal $\nu$ that exists $\sigma$-almost everywhere 
on $\partial \Omega.$ If $\Omega$ is a Lipschitz hypograph, then
$d\sigma(x) = \sqrt{1+\| 
\nabla \gamma (\widehat{x}) \| _{\mathbb R^{d-1}}^2  } d \widehat{x}$ and 
$$
\nu(x) = \frac{  ( - \nabla \gamma (\widehat{x}),1)}  { \sqrt{1+\| 
\nabla \gamma (\widehat{x}) \| _{\mathbb R^{d-1}}^2  } }
$$
for almost every $x\in \partial \Omega$. Suppose now 
that $\Omega$ is a Lipschitz domain. Because 
$\partial \Omega \subset \bigcup _{x\in \partial \Omega}  Q_{\delta , \delta'}(x)$ 
and $\partial \Omega$ is compact, there exist
$x^1, x^2,\ldots,x^n \in \partial \Omega $ such that 
$$
\partial \Omega \subset \bigcup _{j=1}^n Q_{\delta , \delta'}(x^j).
$$ 
It follows that the family $ (W_j) = (Q_{\delta,\delta'}(x^j))$ is a finite 
open cover of $\partial\Omega$, i.e., each $W_j$ is an open subset of 
$\mathbb R^d$ and $\partial \Omega \subseteq  \bigcup _j W_j$. 
Let $(\varphi_j) $ be a partition of unity subordinate to the open cover 
$(W_j)$ of $\partial \Omega$, i.e.,
$\varphi_j \in \mathscr D(W_j)$ and $\sum _j  \varphi_j(x) =1$  for all 
$x\in \partial\Omega$.
The inner product in $H^s(\partial \Omega)$ is then defined by
$$
(u,v)_{s,\partial \Omega} = \sum_j (\varphi_j u, \varphi_j v) _{H^s(\partial \Omega_j)},
$$
where $\Omega_j$ can be transformed to a Lipschitz hypograph by a rigid motion, 
i.e., by a rotation plus a translation, and satisfies
$W_j \cap \Omega = W_j \cap \Omega_j \ \mbox{for each}  \ j$. 
The associated norm will be denoted by
$\| \cdot \|_{s,\partial \Omega}$.
It is interesting to mention that a different choice of 
$(W_j), (\Omega_j)$ and $(\varphi_j)$ would yield the same space 
$H^s(\partial \Omega)$ with an equivalent norm, for $0\leq s \leq 1$. 
For more on the subject we refer the interested reader to \cite{Ad,Dau,Mc}.

The trace operator maps each continuous function $u$ on $\overline{\Omega}$ 
to its restriction onto $\partial\Omega$ 
and may be extended to be a bounded surjective operator, 
denoted by $\Gamma_s$, from $H^s(\Omega)$ to 
$H^{s-\frac{1}{2}}(\partial\Omega)$ for $1/2<s<3/2$  
\cite{Co,Mc}. The range and null space of $\Gamma_s$ 
are given by
$\mathcal R(\Gamma_s)= H^{s-1/2}(\partial \Omega)$ 
and $\mathcal N(\Gamma_s) = H_0^s(\Omega)$, respectively, 
where $H_0^s(\Omega)$ is defined to be the closure in 
$H^s(\Omega)$ of infinitely differentiable functions 
compactly supported in $\Omega$.
For $s=3/2$, this is no longer valid.
For $s>3/2$, the trace operator 
from $H^s(\Omega)$ to $H^1(\partial \Omega)$ is bounded \cite{Co}. 

Let us set $\Gamma=T_1 \Gamma_1$, where $\Gamma_1$ is the trace 
operator from $H^1(\Omega)$ to $H^{1/2}(\partial \Omega)$ and 
$T_1$ is the embedding operator from  $H^{1/2}(\partial \Omega)$  
into $L^2(\partial \Omega)$. According to a classical result 
of Gagliardo \cite{Ga}, we know that
$\mathcal R(\Gamma) = H^{1/2}(\partial \Omega)$. Since $\Gamma_1$ 
is bounded and $T_1$ is compact \cite{N}, 
the trace operator $\Gamma$ from $H^1(\Omega)$ 
to $L^2(\partial \Omega)$ is also compact. 
	
Now, let us induce $H^1(\Omega)$ with the following inner product:
\begin{equation*}
(u,v)_{\partial, \Omega}= \int_{\Omega} \nabla u \nabla v dx 
+ \int_{\partial \Omega} \Gamma u \Gamma v d \sigma 
\quad \forall u,v  \in H^1(\Omega).
\end{equation*}
The associated norm $\|\cdot\|_{\partial, \Omega}$ is given by
\begin{equation*}
\|u\|_{\partial, \Omega}=\left(\|\nabla u\|^2_{0,\Omega}
+\|\Gamma u\|^2_{0,\partial \Omega} \right)^{1/2}
\end{equation*}
and 
$H^{1}(\Omega)$, induced with the inner product $(\cdot,\cdot)_{\partial,\Omega}$, 
is denoted by $H_{\partial}^{1}(\Omega)$. A well-known result
of Ne\v{c}as \cite{N}, asserts that under the condition that $\Omega$ 
is a bounded Lipschitz domain, the norms $\|\cdot\|_{\partial, \Omega}$ 
and $\|\cdot\|_{1, \Omega}$ are equivalent.

The following characterization is useful to prove 
our trace inequalities in Section~\ref{sec:3}.

\begin{lemma}[See Corollary~6.9 of \cite{TCT}] 
\label{lem:1}
Let $\Gamma$ be the trace operator from $H_{\partial}^1(\Omega)$ 
to $L^2(\partial \Omega)$, $\Lambda $ its Moore--Penrose 
inverse, and $\Lambda ^*$ be its adjoint operator. Then, 
for $0\leq s \leq 1$, we have that $H^{s}(\partial \Omega)
= \mathcal H^{s}(\partial \Omega)$ with equivalence of norms, where 
$\mathcal H^{s}(\partial \Omega) 
= \{  (I+\Lambda ^* \Lambda)^{-s}g \ | \ g\in L^2(\partial \Omega)\}$.
\end{lemma}


\section{Main Results}
\label{sec:3}
  
We begin by proving an important equality that, 
together with the trace inequalities
of Theorem~\ref{thm:1}, will be useful 
to prove our harmonic inequality 
of Theorem~\ref{prop:1}.

\begin{theorem}[The Moore--Penrose inverse equality]
\label{prop:perm} 
Let $\Gamma \in {\mathcal B}( H_{\partial}^1(\Omega),L^2(\partial \Omega))$ 
be the trace operator and $\Lambda \in {\mathcal C}(L^2(\partial \Omega), 
H_{\partial}^1(\Omega))$ its Moore--Penrose inverse. Then, for a real $s$, 
the following equality holds:
$$
T_{\Lambda^*} (I+\Lambda \Lambda^*)^{-s} 
= (I+\Lambda^* \Lambda)^{-s}  T_{\Lambda^*}, 
$$
where $T_{\Lambda^*}=\Lambda^*(I+\Lambda\Lambda^*)^{-1/2}
+\Gamma(I+\Lambda\Lambda^*)^{-1/2}$.
\end{theorem}

\begin{proof}
From Lemma~\ref{eq:lemma2.5Cor2.6},
\begin{equation*}
\begin{split}
\mathcal N(T_{\Lambda^*} (I+\Lambda \Lambda^*)^{-s}) 
&= \mathcal N(T_{\Lambda^*})\\
&= \mathcal N(\Lambda^*(I+\Lambda \Lambda^*)^{-1/2})\\
&= \mathcal N( \Lambda^*)\\
&= H_0^1(\Omega)\\
&=\mathcal N((I+\Lambda^* \Lambda)^{-s} T_{\Lambda ^*}),
\end{split}
\end{equation*}
and we have  
$$
T_{\Lambda^*} (I+\Lambda \Lambda^*)^{-s} v  = 0
= (I+\Lambda^* \Lambda)^{-s}  T_{\Lambda^*} v
$$
for all $v\in \mathcal N(\Lambda^*)$. Now let us consider the operator 
$\Gamma^* \Gamma:  H_{\partial}^1(\Omega) \longrightarrow  H_{\partial}^1(\Omega)$,
where $\Gamma^*$ is the adjoint of the trace operator $\Gamma$.
Given the compactness of $\Gamma$, the operator $\Gamma^* \Gamma$ 
is compact and self-adjoint. Then there exists a sequence of pairs 
$(s_k,v_k)_{k\geq 1}$ associated to $\Gamma^* \Gamma$ such that 
$$
\Gamma^*\Gamma v_k = s_k^2 v_k.
$$  
To prove the equality on $\mathcal N(\Lambda^*)^{\perp}$, 
we show that   
$$ 
T_{\Lambda ^*}( I+\Lambda \Lambda^*)^{-s} v_k 
=  (I+\Lambda^* \Lambda)^{-s} T_{\Lambda ^*} v_k.
$$
To this end, let $\Gamma v_k = s_k z_k$. It follows that 
$\Gamma^* z_k = s_k v_k$ and, from Lemma~\ref{lemme:1}, 
$\Gamma v_k = (I+\Lambda ^* \Lambda)^{-1/2} T_{\Lambda ^*} v_k$.
On the other hand,
$$ 
T_{\Lambda^*} v_k = \Lambda ^* (I+\Lambda \Lambda^*)^{-1/2} v_k 
+ \Gamma (I+\Lambda \Lambda^*)^{-1/2} v_k. 
$$
By putting $(I+\Lambda \Lambda ^*)^{-1}v_k = w_k$, we have 
$v_k = w_k +\Lambda \Lambda ^* w_k$ and 
$$
\Gamma ^* \Gamma v_k = s_k^2 v_k 
= \Gamma ^* \Gamma w_k +w_k,
$$
which implies that 
\begin{equation}
\label{eq:poft1:n}
\begin{split}
(I+\Gamma^* \Gamma)^{-1} \Gamma^* \Gamma v_k 
&= s_k^2 (I+\Gamma^* \Gamma)^{-1} v_k \\
&= w_k \\
&= \Gamma^* \Gamma (I+\Gamma^* \Gamma)^{-1} v_k.
\end{split}
\end{equation}
Using Lemma~\ref{eq:lemma2.5Cor2.6}, it follows from \eqref{eq:poft1:n} that 
$$
(I+\Gamma^* \Gamma)^{-1} \Gamma^* \Gamma v_k 
= \Gamma^* \Lambda ^* (I+\Lambda \Lambda^*)^{-1} v_k.
$$
This leads, again from Lemma~\ref{eq:lemma2.5Cor2.6}, to 
\begin{equation*}
\begin{split}
(I+ \Lambda \Lambda ^*)^{-1} v_k 
&= s_k^2 (I+ \Gamma^* \Gamma )^{-1} v_k \\
&= s_k^2(v_k -(I+\Lambda \Lambda^*)^{-1} v_k),
\end{split}
\end{equation*}
so that 
$$ 
(1+s_k^2) (I+\Lambda \Lambda^*)^{-1} v_k 
= s_k^2 v_k. 
$$
Thus,
\begin{equation*}
\begin{split}
(I+ \Lambda \Lambda^*)^{-1} v_k 
&= \frac{s_k^2}{1+s_k^2}
v_k \\
&= (1+s_k^2)^{-1} s_k^2 v_k,
\end{split}
\end{equation*}
which implies that  
$$ 
\left(I+ \Lambda \Lambda^*\right)^{-s} v_k 
= \left(s_k^2\left(1+s_k^2\right)^{-1}\right)^{s} v_k.
$$
In particular,
$$ 
(I+ \Lambda \Lambda^*)^{-1/2} v_k 
= s_k(1+s_k^2)^{-1/2} v_k.
$$
Consequently,
\begin{equation*}
\begin{split}
T_{\Lambda^*} v_k 
&=  \Lambda^*(I+\Lambda \Lambda^*)^{-1/2} v_k 
+ \Gamma(I+\Lambda \Lambda ^*)^{-1/2} v_k \\
&= \Lambda ^* \left(  \frac{s_k}{\sqrt{1+s_k^2}} v_k \right)
+ \Gamma \left( \frac{s_k}{\sqrt{1+s_k^2}} v_k \right) \\
&= \frac{1}{\sqrt{1+s_k^2}} z_k + \frac{s_k^2}{\sqrt{1+s_k^2}} z_k \\
&= \sqrt{1+s_k^2} z_k.
\end{split}
\end{equation*}
Therefore,
\begin{equation*}
\begin{split}
T_{\Lambda^*}(I+\Lambda \Lambda^*)^{-s} v_k 
&= \left( \frac{s_k^2} {1+s_k^2}\right)^s 
\left( \frac{1+s_k^2}{\sqrt{1+s_k^2}}\right) z_k \\
&= \left(\frac{s_k^2} {1+s_k^2} \right)^s \sqrt{1+s_k^2} z_k.
\end{split}
\end{equation*}
On the other hand, 
$\Gamma \Gamma^* z_k = \Gamma s_k v_k = s_k^2 z_k$.
By putting  
$(I+\Lambda^*  \Lambda)^{-1} z_k = e_k$, we have 
$$ 
z_k = e_k+\Lambda ^* \Lambda e_k,
$$
which implies that 
$$ 
\Gamma \Gamma^* z_k = s_k^2 z_k = \Gamma \Gamma^* e_k +e_k 
$$
and 
\begin{equation*}
\begin{split}
(I+\Gamma \Gamma^*)^{-1} \Gamma \Gamma^* z_k 
&= s_k^2 (I+\Gamma \Gamma^*)^{-1} z_k \\
&=e_k\\
&=\Gamma \Gamma^* (I+\Gamma \Gamma^*)^{-1} z_k. 
\end{split}
\end{equation*}
Using Lemma~\ref{eq:lemma2.5Cor2.6}, it follows that 
$$
(I+\Gamma \Gamma^*)^{-1} \Gamma \Gamma^* z_k = \Gamma \Lambda (I+\Lambda^*  \Lambda)^{-1} z_k.
$$ 
This leads, again from Lemma~\ref{eq:lemma2.5Cor2.6}, to 
$$
(I+\Lambda^*  \Lambda)^{-1} z_k = s_k ^2(I+\Gamma \Gamma ^*)^{-1} z_k 
= s_k^2 \left(z_k - (I+\Lambda^*  \Lambda)^{-1} z_k \right),
$$
so that 
$$
(1+s_k^2) (I+\Lambda^*  \Lambda)^{-1}z_k = s_k^2 z_k 
$$
and
$$    
(I+\Lambda^*  \Lambda)^{-1} z_k = s_k^2 (1+s_k^2) ^{-1} z_k. 
$$
Consequently, 
$$
\left(I+\Lambda^*  \Lambda\right)^{-s} z_k 
= \left(s_k^2 \left(1+s_k^2\right)^{-1}\right)^s z_k, 
$$
which implies that 
\begin{equation*}
\begin{split}
\left( I+ \Lambda^* \Lambda \right) ^{-s} T_{\Lambda^*} v_k 
&=  \left( I+\Lambda ^* \Lambda \right) ^{-s} \sqrt{1+s_k^2} z_k \\
&= \sqrt{1+s_k^2} \left( \frac{s_k^2}{1+s_k^2} \right) ^s z_k. 
\end{split}
\end{equation*}
Hence, one has 
$ \left( I+ \Lambda^* \Lambda \right) ^{-s} T_{\Lambda^*} v_k  
=  T_{\Lambda^*}\left( I+ \Lambda \Lambda^* \right) ^{-s} v_k$
for all $k\geq 1$ and the proof is complete.
\end{proof}
    
Let us now consider the family of Hilbert spaces  
$$ 
\mathcal H^s(\Omega) = \{  v\in H^s(\Omega) \ / \ \Delta v =0 
\ \text{ in } \ \mathscr{D}' (\Omega) \},
\quad s\geq 0,  
$$ 
that consist of real harmonic functions on the usual Sobolev space 
$H^s(\Omega)$. For $1 < s < 3/2$, we equip $\mathcal H^s(\Omega)$ 
with the following  norm:
$$  
\|u\|_{\mathcal H^s(\Omega)}  
=  \|\Gamma_s u\|_{s-1/2,\partial \Omega}.
$$
      
\begin{theorem}[The trace inequalities]
\label{thm:1}
Let $\Omega\subset \mathbb{R}^d$, $d\geq 2$, be a bounded Lipschitz 
domain with boundary $\partial \Omega$. Consider, for $1 < s<3/2$, 
the trace operators $\Gamma_s$ and $\Gamma$ from $H^s(\Omega)$ 
to $H^{s-1/2}(\partial \Omega)$ and from 
$H_{\partial}^1(\Omega)$ to $L^2(\partial \Omega)$, respectively, 
and let $\Lambda$ be the Moore--Penrose inverse of $\Gamma$.
Then there exist two positive constants $c_1$ and $c_2$ such that 
the inequalities
\begin{equation}
\label{eq:trace:ineq}
c_1 \|\Gamma_s v\| _{s-1/2 ,\partial \Omega}
\leq \|(I+\Lambda^* \Lambda)^{s-1/2}\Gamma \tilde{v}\|_{0,\partial \Omega} 
\leq  c_2 \|\Gamma_s v\|_{s-1/2,\partial \Omega}
\end{equation}
hold for all $v\in \mathcal H^s(\Omega)$, where $\tilde{v}$ 
is the embedding of $v$ in $\mathcal H^1(\Omega)$.
\end{theorem}

\begin{proof}
Assume $1 < s<3/2$ and let $v\in \mathcal H^s(\Omega)$ and 
$\tilde{v}$ be its embedding in $\mathcal H^1(\Omega)$. Clearly, 
$\Gamma_s v \in H^{s-1/2}(\partial \Omega)$ and
$\Gamma \tilde{v} \in  L^2(\partial \Omega)$.  
From Lemma~\ref{lem:1}, it follows that
$$ 
\mathcal H ^{s-1/2}(\partial \Omega)= H^{s-1/2}(\partial \Omega)
$$ 
with equivalence between the norm
$\|\cdot\|_{s-1/2,\partial \Omega}$ 
and the graph norm defined for $g\in L^2(\partial \Omega)$ by
$$ 
g\longmapsto \|(I+\Lambda^* \Lambda)^{s-1/2}g \|_{0,\partial \Omega}.
$$ 
Equivalently, there exist two positive constants $c_1$ and $c_2$ 
such that \eqref{eq:trace:ineq} holds for all 
$v\in \mathcal H^s(\Omega)$.
\end{proof}

As an application of Theorems~\ref{prop:perm} and \ref{thm:1}, 
we prove the harmonic inequality \eqref{eq:Harm:ineq:s:1:3/2}.

\begin{theorem}[The harmonic inequality for $\mathbf{1< s<3/2}$]
\label{prop:1}
Assume $1< s<3/2$. Then, for all $v\in \mathcal H^s(\Omega)$,  
the following inequality holds:
\begin{equation} 
\label{eq:Harm:ineq:s:1:3/2}
\|v\|_{\mathcal H^s(\Omega)} 
\leq 
\|T_ {\Lambda^*}\| \  \|(I+\Lambda \Lambda^*)^{s-1} 
\widetilde{ v} \| _{\partial,  \Omega}, 
\end{equation} 
where 
\begin{equation*}
T_{\Lambda^*}=\Lambda^*(I+\Lambda \Lambda^*)^{-1/2}
+\Gamma(I+\Lambda \Lambda^*)^{-1/2}
\end{equation*}
and $ \widetilde{v}$ is the embedding of $v$ in $\mathcal{H}^1(\Omega)$.
\end{theorem}

\begin{proof}
Consider $v\in \mathcal H^s(\Omega)$ and  $\widetilde{v}$  
its embedding in $\mathcal{H}^1(\Omega)$. It follows from Theorem~\ref{thm:1} 
that there exist a positive constant $c_3$ such that   
$$
\|v\|_{\mathcal H^s(\Omega)}=\|\Gamma_s v\|_{s-1/2,\partial \Omega}  
\leq 
c_3 \ \|(I+\Lambda^* \Lambda)^{s-1/2} \Gamma \widetilde{ v} \| _{0, \partial \Omega}.
$$
Moreover, from Lemma~\ref{lemme:1}, we have
$$
\Gamma = (I+\Lambda^* \Lambda)^{-1/2} T_{\Lambda^*},
$$
where   
\begin{equation*}
T_{\Lambda^*}=\Lambda^*(I+\Lambda \Lambda^*)^{-1/2}
+\Gamma(I+\Lambda \Lambda^*)^{-1/2}.
\end{equation*}
It follows from 
Theorem~\ref{prop:perm} that
$$
\|v\|_{\mathcal H^s(\Omega)}
= \|(I+\Lambda^* \Lambda)^{s-1} T_{\Lambda^*} 
\widetilde{ v} \|_{0,\partial \Omega} 
=  \|  T_{\Lambda^*} (I+\Lambda \Lambda^*)^{s-1} 
\widetilde{ v} \| _{0,\partial \Omega}.
$$
Lemma~\ref{thm:3.5} asserts that
$T_{\Lambda^*}$ is bounded and
the intended inequality \eqref{eq:Harm:ineq:s:1:3/2}
follows.
\end{proof}

Now consider the embedding operator $E$ from $H^1(\Omega)$ 
into $L^2( \Omega)$ and its adjoint $E^*$, which is 
the solution operator of the following 
Robin problem for the Poisson equation:
\begin{equation*}
\begin{cases}
-\Delta u=f & \text{ in }  \Omega, \\
\partial_{\nu}u+\Gamma u =0 
& \text{ on }  \partial \Omega,
\end{cases}
\end{equation*}
where $f\in L^2(\Omega)$ and $\partial_{\nu}$ denotes 
the normal derivative operator with exterior normal $\nu$.
Let $E_0^*$ be the solution operator of the Dirichlet 
problem for the following Poisson equation:
\begin{equation*}
\begin{cases}
-\Delta u^0=f & \text{ in }  \Omega, \\
\Gamma u^0=0 & \text{ on }  \partial \Omega.
\end{cases}
\end{equation*}
By setting $E_1^*= E^*-E_0^*$ and $u^1 =E_1^*f$, 
it follows that $u^1$ is the solution of the 
Dirichlet problem for the following Laplace equation:
\begin{equation*}
\begin{cases}
-\Delta u^1=0 & \text{ in }  \Omega,\\
\Gamma u^1=\Gamma u & \text{ on }  \partial \Omega.
\end{cases}
\end{equation*}  
Let $0\leq s \leq 1$, $F_1$ be the Moore--Penrose inverse of $E_1$, 
$F_1^*$ be its adjoint operator, and denote 
$$
\mathcal X^s(\Omega) = \left\{ \  (I+F_1^*F_1)^{-s/2} v  \ | \ v \in \mathcal H(\Omega) \ \right\},
$$
where $\mathcal H(\Omega)$ is the Bergman space.
Next we prove the harmonic inequalities for 
the case $s = 1$.

\bigskip

\begin{theorem}[The harmonic inequalities for $\mathbf{s = 1}$]
\label{prop:3}
Let $\Omega\subset \mathbb{R}^d$, $d\geq 2$, 
be a bounded Lipschitz domain. Then, for all $ v\in \mathcal H^1(\Omega)$, 
there exist two positive constants $c_1^{\prime}$ and $c_2^{\prime}$, 
not depending on $v$, such that
\begin{equation} 
\label{eq:Harm:ineqs:s:0:1} 
c_1^{\prime}  \|v\|_{\partial,\Omega} 
\leq  
\|(I+F_1^* F_1)^{1/2} E_1  v\| _{0,\Omega}  
\leq c_2^{\prime} \|v\|_{\partial,\Omega}.
\end{equation} 
\end{theorem}

\begin{proof} 
From Lemma~\ref{lemme:1}, the decomposition 
$$
E_1=(I+F_1^* F_1)^{-1/2}T_{F_1^*}
$$
holds, where 
$$
T_{F_1^*}  = F_1^*(I+F_1F_1^*)^{-1/2} + E_1(I+F_1F_1^*)^{-1/2}.
$$ 
Moreover, since $\mathcal{R}(E_1)=\mathcal H^1(\Omega)$, 
it follows that $\mathcal H^1(\Omega)=\mathcal X^1(\Omega)$. 
Now consider the graph norm
$$
\|u\|_{\mathcal X^1(\Omega) }
= \|(I+F_1^* F_1)^{1/2}   u\| _{0,\Omega}
$$
for $u\in \mathcal X^1(\Omega)$. 
It follows that $E_1 v \in \mathcal X^1(\Omega)$ 
for $v\in \mathcal H^1(\Omega)$
and
$$ 
\|(I+F_1^* F_1)^{1/2} E_1 v\| _{0,\Omega} 
= \|T_{F_1^*} v\|_{0,\Omega}. 
$$
In agreement with Lemma~\ref{lem:2}, we
can view $T_{F_1^*} $ as an isomorphism from 
$\mathcal H^1(\Omega)$ into $\mathcal H(\Omega)$, 
and there exist two positive 
constants $c_1^{\prime}$ and $c_2^{\prime}$, 
not depending on $v$, such that
\eqref{eq:Harm:ineqs:s:0:1} holds.
\end{proof}

The harmonic inequalities of Theorem~\ref{prop:3} 
are a useful tool to provide a functional characterization 
of the harmonic Hilbert spaces for the range of values 
$0\leq s\leq 1$. 

\begin{corollary} 
\label{cor:spect:caract}
Assume $0\leq s \leq 1$. Then $\mathcal H^s(\Omega)$ form 
an interpolatory family. Moreover,
$$ 
\mathcal H^s(\Omega) =  \mathcal X^s(\Omega)
$$  
with equivalence of norms.
\end{corollary}

\begin{proof}
For $s=0$, one has the equality 
$\mathcal X(\Omega)= \mathcal H(\Omega)$
by definition. For $s=1$, Theorem~\ref{prop:3} asserts that 
$\mathcal X^1(\Omega) = \mathcal H^1(\Omega)$ with the equivalence of the norm 
$\|\cdot\|_{\mathcal X^1(\Omega) }$ with the norm on $\mathcal H^1(\Omega)$, 
which is the same as the one on $H_{\partial}^1(\Omega)$. The intended equality 
$\mathcal H^s(\Omega) = \mathcal X^s(\Omega)$ with equivalence of norms,
$0<s<1$, follows from classical results on the theory of positive self-adjoint operators, 
which assert that both $\mathcal X^s(\Omega)$ and $\mathcal H^s(\Omega)$
form an interpolating family for $0<s<1$.
\end{proof}


\begin{acknowledgement}
This research is part of first author's Ph.D. project, 
which is carried out at Moulay Ismail University, Meknes. 
It was essentially finished during a visit of Touhami 
to the Department of Mathematics of University of Aveiro, 
Portugal, November 2018. The hospitality of the host 
institution and the financial support of Moulay Ismail 
University, Morocco, and CIDMA, Portugal, are here 
gratefully acknowledged. Torres was partially supported 
by the Portuguese Foundation for Science and Technology (FCT) 
through CIDMA, project UID/MAT/04106/2019.
\end{acknowledgement}



\end{document}